\documentclass[10pt]{article}
\usepackage{graphicx,amsthm}
\usepackage{amsfonts,amssymb,amscd,amsmath,enumerate,verbatim,latexsym,epsfig,amsfonts}

\newtheorem{Theorem}{Theorem}[section]
\newtheorem{Lemma}[Theorem]{Lemma}
\newtheorem{Corollary}[Theorem]{Corollary}
\newtheorem{Proposition}[Theorem]{Proposition}

\newtheorem{Example}[Theorem]{Example}

\newtheorem{Definition}[Theorem]{Definition}

\newtheorem{Question}[Theorem]{Question}

\begin{document}
\title{Power Graph and Exchange Property for Resolving Sets}
\author{U. Ali\thanks{Corresponding author.}, G. Abbas, S. A. Bokhary
\vspace{4mm}\\
\normalsize  Centre for Advanced Studies in Pure and Applied Mathematics,\\\normalsize
Bahauddin Zakariya University, Multan, Pakistan\\\normalsize
{\tt uali@bzu.edu.pk, sihtsham@gmail.com, k\_abbas2008@yahoo.com}}

\date{}
\maketitle


\begin{abstract}
 A formula for computing the metric dimension of a simple graph, having no singleton twin, is given. A sufficient condition for a simple graph to have the exchange property, for resolving sets, is found. Some families
of power graphs of finite groups, having this exchange property, are
identified. The metric dimension of the power graph of a dihedral group is also
computed.
\end{abstract}



\section{Introduction}
 Let $G$ be a finite group. An \emph{undirected power graph}
$\mathcal{P}_{G}$, associated to $G$, is a graph whose
 vertices are the elements of $G$, and
there is an edge between two vertices $x$ and $y$ if either $x^m=y$
or $y^m=x$, for some positive integer $m$. The power digraph of $G$ is the
digraph $\overrightarrow{\mathcal{P}_{G}}$ with the vertex set $G$, and there is an arc from vertex
$x$ to $y$ if $x^m=y$, for some positive integer $m$.
The power digraphs were considered in \cite{Kelarev1, Kelarev2,
Kelarev3}. Motivated by this, Chakrabarty, Ghosh, and Sen (see
\cite{Chakrabarty}) studied undirected power graphs of
semigroups. Recently, many interesting results on the power graph of a finite group
have been obtained (see \cite{Cameron}, \cite{Cameron1}, \cite{Mirzargar},
\cite{Moghaddamfar}, \cite{Chelvam}). A power graph is always connected. For other results and open questions, we refer the survey~\cite{Abawajy}.

Let $ \Gamma$ be a finite, simple, and connected
graph with vertex set $V(\Gamma)$ and edge set $ E(\Gamma)$. The \emph{distance} $d_{\Gamma}(u,v)$ between two vertices
$u,v\in V(\Gamma)$ is the length of the shortest path between them.
Let $W=\{w_{1},w_{2},\dots ,$ $w_{k}\}$ be an ordered set of
vertices of $\Gamma$, and let $v\in V(\Gamma)$. Then, the
\emph{representation} of $v$ with respect to $W$ is the
$k$-tuple $(d_{\Gamma}(v,w_{1}),d_{\Gamma}(v,w_{2}),\ldots
,d_{\Gamma}(v,w_{k}))$. Two vertices $u,v\in V(\Gamma)$ are resolved by $W$ if they have different representations.
$W$ is \textit{resolving set} or \emph{locating set} if every
vertex of ${\Gamma}$ is, uniquely, identified by its distances from
the vertices of $W$. Thus, in a resolving set, every vertex of
${\Gamma}$ has distinct representations. A
resolving set of minimum cardinality is called a \textit{basis} for
${\Gamma}$. The cardinality of such a resolving set is called the \emph{metric dimension} of $\Gamma$, and is denoted by $\beta(\Gamma)$ (see \cite{Buczkowski}, \cite{Chartrand3}, \cite{Slater1}, \cite{Imran2}, \cite{tomescu},
\cite{jav}). A resolving set is called \emph{minimal} if it contains no
resolving set as a proper subset. As an application, S. Khuller ~\cite{khuller} considered the metric dimension and basis of a connected
graph in robot navigation problems.

Whenever $W_1$ and $W_2$ are any two minimal resolving sets for
$\Gamma$, and for every $u\in W_1$ there is a vertex $v\in W_2$ such that
$(W_1\setminus\{u\})\cup\{v\}$ is also a minimal resolving set. Then, Resolving sets are said to have \emph{the exchange property} in the graph
$\Gamma$ (for detail, see~\cite{Boutin}). All the graphs considered in this paper are finite, simple, and connected. Also, all the groups considered are finite. Furthermore, the
exchange property of a graph $\Gamma$ always means the property for resolving sets.

The \emph{open neighborhood} of a vertex $u\in V(\Gamma)$, denoted by $N(u)$, is the set
$$\{v\in V(\Gamma): d_{\Gamma}(u,v)=1\},$$
and the \emph{closed neighborhood} of $u\in V(\Gamma)$, denoted by
$N[u]$, is the set
$$\{v\in V(\Gamma): d_{\Gamma}(u,v)=1\}\cup\{u\}.$$
The two vertices $u$ and $v$ in a graph $\Gamma$ are called twins, denoted by $u\equiv v$, if either $N[u]=N[v]$ or $N(u)=N(v)$. The relation $\equiv$ is an
 equivalence relation (see~\cite{Hernando}). Also, $d_{\Gamma}(u_1,w)=d_{\Gamma}(u_2,w)$ for $u_1\equiv u_2$, and for all $w\in V(\Gamma)$. Let $\overline{u}$ denote the twin class of $u$ with respect to the relation $``\equiv"$, and let $\mathcal{U}(\Gamma)=\{{\overline{u}|u \in \Gamma}\}$ be the set of all such twin classes.
  \begin{Definition}
A vertex $u$ is called singleton twin if $\overline{u}=\{u\}$.
  \end{Definition}
Our first result give a formula to compute the metric dimension of a graph, having no singleton twin.
\begin{Theorem}\label{alsinex}
Let $\Gamma$ be a graph, having no singleton twin. Let there are $n$ non-singleton twins, each of size $m_i$. Then, $$\beta(\Gamma)=\sum_{i=1}^n m_i-n.$$
Moreover, every minimal resolving set is a basis of the graph $\Gamma$.
\end{Theorem}
Our second result provides a sufficient condition for a graph to have
the exchange property.
\begin{Theorem}\label{nosinex}
 A graph $\Gamma$, having no singleton twin, has the exchange property.
\end{Theorem}
In our next theorem, the metric dimension of the power graph, associated to a dihedral group $D_{2n}$
of order $2n$, is computed.
\begin{Theorem}\label{mdDi}
$\beta(\mathcal{P}_{D_{2n}})=\beta(\mathcal{P}_{Z_{n}})+n-2,$ where $Z_{n}$ is a cyclic group of order $n$.
\end{Theorem}
In the following theorem, we identified some finite groups whose corresponding power graph
has the exchange property.
\begin{Theorem}\label{Expro}
Let $G$ be a finite group and $\mathcal{P}_{G}$ be the power graph associated to $G$. Then, the exchange
property holds in $\mathcal{P}_{G}$ if one of the following conditions is satisfied:\\
\emph{(i)} $G$ is cyclic and $| G|=2k+1$, for positive integers $k.$\\
\emph{(ii)} $| G|$ is a power of a prime number $p$, and $G$ is
abelian.
\end{Theorem}

The rest of the sections are organized as follows:\\
In section 2, the exchange property is discussed; also, proofs of Theorem~\ref{alsinex} and~\ref{nosinex} are
given. In section 3, Theorems ~\ref{mdDi} and~\ref{Expro} are proved.
 \section{The Exchange Property for Resolving Sets}
 Every vector in a finite dimensional vector space is, uniquely, determined (written as a
linear combination) by the elements of a basis of the vector space. A basis of a vector space has the exchange
property. Similarly, each vertex of a finite graph can be,
uniquely, identified by the vertices of a resolving set. Therefore, resolving sets of a finite graph behave
like bases in a finite dimensional vector space. Unlike bases of a vector space, the resolving sets do not always
have the exchange property. The following results about the exchange property for
different graphs can be found in the literature.
\begin{Theorem} [\cite{Boutin}, Theorem 3]
The exchange property holds for resolving sets in trees.
\end{Theorem}
\begin{Theorem} [\cite{Boutin}, Theorem 7]
For $n\geq 8$, the resolving sets do not have the exchange property in a
wheel graph $W_{n}$.
\end{Theorem}
\begin{Theorem}  [\cite{tomescu1}, Theorem 5]
For $n\geq 4$ and $n$ is even, the nacklace graph
$Ne_{n}$ does not have the exchange property.
\end{Theorem}

\begin{Lemma}\label{verclass}
Let $W$ be a resolving set for a graph $\Gamma$, and $v_1,v_2\in V(\Gamma)$. Then, either $v_1\in W$ or $v_2\in W$ for  $v_1\equiv v_2$ .
\end{Lemma}
\begin{proof}
For $v_1\equiv v_2$, we have $d_{\Gamma}(v_{1},u)$=$d_{\Gamma}(v_{2},u)$ for all $u \in
V({\Gamma})\setminus \{v_1,v_2\}$. Therefore, $v_{1}$ and $v_{2}$ can not be part of  $V({\Gamma})\setminus W$. Otherwise, $v_1$ and $v_2$ remain unresolved.
\end{proof}

\emph{\textbf{Proof of the Theorem\ref{alsinex}}:}
Let $W$ be a basis of $\Gamma$. By Lemma~\ref{verclass}, $W$ contains $m_i-1$ vertices of each twin class of size $m_i$.  Now, let $u$ and $v$ be two vertices which are not twins. Then, there must be some $w\in W$ such that $d_{\Gamma}(u,w)\neq d_{\Gamma}(v,w)$; otherwise $d_{\Gamma}(u,x)=d_{\Gamma}(v,x)$ for all $x\in V(\Gamma)$, which means that $u$ and $v$ twins, a contradiction. Consequently, exactly one representative, from each twin class, stay outside $W$. Therefore, $\beta(\Gamma)=\sum_{i=1}^n m_i-n.$\\
 The cardinality of a minimal resolving set $W_1$ is $\geq \beta(\Gamma).$ Now, $W_1$ must have exactly $\beta(\Gamma)=\sum_{i=1}^n m_i-n$ vertices. Otherwise, $W_1$ contains an entire twin class $\overline{u}$ of a vertex $u$, and $W_1\setminus \{u\}$ is again resolving set, a contradiction. Therefore, every minimal resolving set is a basis.

\emph{\textbf{Proof of the Theorem\ref{nosinex}}:}
   Let $W_{1}$ and $W_{2}$ be two different minimal resolving sets in the graph $\Gamma$, and  let $u_{1}\in W_1$. If $u_1\in W_2$, then obviously $(W_1\setminus \{u_{1}\})\cup \{u_{1}\}$ is a minimal resolving set. Let $u_1\notin W_2$. There exists a vertex $u_2\notin W_1$ such that $ u_{1}\equiv u_2$. Otherwise, $W_1$ contains an entire twin class, and $W_1$ is not minimal by Theorem~\ref{alsinex}, a contradiction. Now, by Lemma~\ref{verclass}, $ u_{2}\in W_2$, and every vertex in $V(\Gamma)\setminus \{u_1,u_2\}$ is at the same distance from the vertices $u_1$ and $u_2$. Therefore, the vertices which are resolved by $u_1$ are also resolved by $u_2$ and vice versa. Consequently, $(W_1\setminus \{u_{1}\})\cup \{u_{2}\}$ is again a minimal resolving set. \hfill $\Box$
\section{Power Graph of Finite groups}
\begin{Proposition} [\cite{Cameron1}, Proposition 2]\label{samnbhd} Suppose $x$ and
$y$ are elements of an abelian group $G$, then $x$ and $y$ have the
same closed neighborhoods, in the power graph $\mathcal{P}_{G}$, if
and only if one of the following holds:\\
\emph{(i)} $\langle x \rangle=\langle y \rangle;$\\
\emph{(ii)} $G$ is cyclic, and one of $x$ and $y$ is a generator of
$G$ and the other is the\\ \indent identity $e$; and\\
\emph{(iii)} $G$ is cyclic of prime order ($x$ and $y$ are
arbitrary).
\end{Proposition}
\begin{Definition}\label{nnn} \emph{(\cite{MF})}
 For elements $x$ and $y$ in a group $G$, write
 $R\{x,y\}=\{z:z\in V(\mathcal{P}_{G}), d_{\mathcal{P}_{G}}(x,z)\neq d_{\mathcal{P}_{G}}(y,z)\}$.
  \end{Definition}
An involution is a non-identity element of order 2 of a group $G$. A resolving involution, in the power graph $\mathcal{P}_{G}$ of a group $G$, is an involution $w$ satisfies that, there exist two vertices $ x,y\in V(\mathcal{P}_{G})\setminus\overline{w}$ with
$R\{x,y\}$=$\{x,y,w\}$. let $W(\mathcal{P}_{G})$ denotes the set of all resolving involutions of $\mathcal{P}_{G}$.
\begin{Example}
Let $G=\{e,x,x^{2},x^{3},x^{4},x^{5}\}$ be the cyclic group of order 6. Note that
$R\{x,y\}=\{u,v,x^{3}\}$, for $u\in \{x,x^{5}\}$, and
$v\in\{x^{2},x^{4}\}$. Therefore, $x^{3}$ is a resolving involution of
of $\mathcal{P}_{G}$.
\end{Example}
Let $\Psi$ denote the set of noncyclic groups $G$ such that there
exists an odd prime $p$ such that the following conditions hold (see~\cite{MF}):\\
(C1) the prime divisors of $| G |$ are $2$ and $p$;\\
(C2) the subgroup of order $p$ is unique;\\
(C3) there is no element of order $4$ in $G$; and\\
(C4) each involution of $G$ is contained in a cyclic subgroup of
order $2p$.

In the original paper~\cite{MF}, for a finite group $G$, the notations $| G
|$; $| \mathcal{U}(G) |$; and $ | W(G) |$ were used for $| V(\mathcal{P}_{G})
|$; $| \mathcal{U}(\mathcal{P}_{G}) |$; and $ | W(\mathcal{P}_{G}) |$ respectively. We give the following results in our notations.
\begin{Theorem} \label{mdpg}\emph{(\cite{MF}, Theorem 3.23)}\\\\
\emph{(i)} If $G\in \Psi,$ then $$\beta(\mathcal{P}_{G})=| V(\mathcal{P}_{G})
|-|\mathcal{U}(\mathcal{P}_{G}) | +1.$$\\
\emph{(ii)} If $G\notin \Psi,$ then $$\beta(\mathcal{P}_{G})=| V(\mathcal{P}_{G})
|-| \mathcal{U}(\mathcal{P}_{G}) | + | W(\mathcal{P}_{G}) |.$$
\end{Theorem}
\begin{Corollary} [\cite{MF}, Corollary 3.24] \label{mdcyc}  Suppose that $n= p^{r_{1}}_{1}\cdots p^{r_{t}}_{t}$, where $p_{1},\ldots,p_{t}$ are primes with
 $p_{1}<\cdots <p_{t}$,
and $r_{1},\ldots,r_{t}$ are positive integers. Let $Z_{n}$ denote
the cyclic group of order $n$. Then
\begin{center}
$\beta(\mathcal{P}_{Z_{n}})=\left\{\begin{array}{ll}
                             n-1, & \mbox{if } t=1 ; \\
                             n-2r_{2}, & \mbox{if }(t,p_{1},t_{1})=(2,2,1) ;\\
                             n-2r_{1}, & \mbox{if }(t,p_{1},t_{2})=(2,2,1) ;\\
                             n+1-\prod\limits^t_{i=1}(r_{i}+1), & \mbox{otherwise.}
                           \end{array}
                         \right.$
\end{center}
\end{Corollary}

A dihedral group is presented as:
$$D_{2n}=\langle a,b| a^n=b^2=e, (ab)^2=e\rangle.$$
$D_{2n}$ is the disjoint union of the cyclic subgroup $
Z_{n}\cong \langle a\rangle=\{e, a, a^2, \ldots, a^{n-1}\}$, and the set of involutions
$B=\{b, ab, a^2b, \ldots, a^{n-1}b\}$.

\begin{Lemma}\label{dihinv}
Let $w\in B$, then, in the power graph $\mathcal{P}_{D_{2n}}$, the following are true:\\
\emph{(i)} $\overline{w}=B;$\\
 \emph{(ii)} $w$ is not a resolving involution.
\end{Lemma}
\proof The neighborhood, in the graph $\mathcal{P}_{D_{2n}}$, of every involution $w\in B$ is $\{e\}$. Therefore, $\overline{w}=B$. If $x,y\in V(\mathcal{P}_{D_{2n}})\setminus B$, then there are two possibilities:\\
1) $ x=a^{s} $, $ y=e$, $ 1\leq s \leq n-1$;\\
2) $ x=a^{s_1} $, $ y=a^{s_2} $, $ 1\leq s_1,s_2 \leq n-1$.\\
 In the above two cases, one can see that
$R\{x,y\}\neq \{x,y,w\}$. Therefore, $w\in B$ is not a resolving involution.
\endproof

\emph{\textbf{Proof of the Theorem\ref{mdDi}}:} By part (ii) of
Lemma~\ref{dihinv}, every resolving involution in $\mathcal{P}_{D_{2n}}$ belongs
to the subgraph $\mathcal{P}_{\langle a \rangle}$, corresponding to the cyclic subgroup $\langle a \rangle$, of $D_{2n}$. Therefore,
$W(\mathcal{P}_{D_{2n}})=W(\mathcal{P}_{\langle a \rangle})$. In the subgraph $\mathcal{P}_{\langle a \rangle}$, the
identity $e$ and the generator $a$ are twins. However, $e$
is the unique singleton twin in $\mathcal{P}_{D_{2n}}$. By part (i) of Lemma~\ref{dihinv}, all $w\in B$  are in the same twin class. Therefore, the
set $\mathcal{U}(\mathcal{P}_{D_{2n}})$ is the disjoint union of
$\mathcal{U}(\mathcal{P}_{\langle a \rangle})$; the twin class of $e$; and the twin class of $w,$
for $w\in B$. Consequently, $|\mathcal{U}(\mathcal{P}_{D_{2n}})|=|\mathcal{U}(\mathcal{P}_{\langle a \rangle})|+2.$
A dihedral group $D_{2n}$ does not satisfy the condition (C4);
therefore, $D_{2n}\notin \Psi$. Now, put  $| V(\mathcal{P}_{D_{2n}})|=| V(\mathcal{P}_{\langle a \rangle})|+n$; $|\mathcal{U}(\mathcal{P}_{D_{2n}})|=|\mathcal{U}(\mathcal{P}_{\langle a \rangle})|+2$; and $| W(\mathcal{P}_{D_{2n}})|=| W(\mathcal{P}_{\langle a \rangle})|$ in the equation of part (ii) of Theorem~\ref{mdpg} to complete the proof.\hfill
$\Box$

To compute the exact value of $\beta(\mathcal{P}_{D_{2n}})$, one can use Theorem~\ref{mdDi} and Corollary~\ref{mdcyc}.
\begin{Lemma}\label{sinin}
A singleton twin $x$, in the power graph $\mathcal{P}_{G}$, is either an involution or
the identity $e$ in the group $G$.
\end{Lemma}
\proof If $x$ in $G$ is not an involution or $e$, then the order $o(x)$
is $\geq 3$ and $N[x]=N[x^{-1}],$ a contradiction.
\endproof
 \emph{\textbf{Proof of the Theorem\ref{Expro}}:} Let $G$ be a cyclic group of odd order, and $y$ is a generator of $G$. Then, there is no involution in the group $G$. Also, part (ii) of Proposition~\ref{samnbhd} implies that $\overline{y}=\overline{e}$, and $e$ is not a singleton twin.
  Therefore, by Lemma~\ref{sinin}, the graph $\mathcal{P}_{G}$ has no singleton twin. Hence, by Theorem~\ref{nosinex}, the exchange property holds in $\mathcal{P}_{G}$.\\
 Let $G$ be an abelian group of order $|G|=p^m$, for
some prime $p$, then $\mathcal{P}_{G}$ is a complete graph. Therefore, $\mathcal{P}_{G}$ has no singleton twin, and has the exchange property by Theorem~\ref{nosinex}. \\

The following example shows that the converse of Theorem \ref{nosinex} and Theorem \ref{Expro} is not true.
\begin{Example}[]
Let $\mathcal{P}_{Z_6}$ be the power graph, where $Z_6\cong \langle x
\rangle=\{e,x,x^{2},x^{3},x^{4},x^{5}\}$. Then, the order of the group is even and not a power of a prime. Furthermore, the power graph contains the singleton twin $x^3$. Still, the graph has the exchange property for resolving sets.
\end{Example}
We did not encounter a power graph of a finite group which
does not have the exchange property. Therefore, the following
question make sense to be posed.
\begin{Question}
Does there exist a finite group whose power graph does not hold the
exchange property?
\end{Question}

\end{document}